\documentclass[11pt]{amsart}
\usepackage{amscd}
\usepackage{amsmath}
\usepackage{amsxtra}
\usepackage{amsfonts}
\usepackage{amssymb}
\usepackage{setspace}
\usepackage{color}

\newtheorem{theorem}{Theorem}[section]

\newtheorem{lemma}[theorem]{Lemma}
\newtheorem{proposition}[theorem]{Proposition}

\theoremstyle{definition}
\newtheorem{definition}[theorem]{Definition}
\newtheorem{remark}[theorem]{Remark}

\newtheorem{example}[theorem]{Example}
\theoremstyle{remark}

\renewcommand{\theclaim}{\textup{\theclaim}}

\newtheorem*{acknowledgements}{Acknowledgements}

\numberwithin{equation}{section}

\def\openone

{\mathchoice

{\hbox{\upshape \small1\kern-3.3pt\normalsize1}}

{\hbox{\upshape \small1\kern-3.3pt\normalsize1}}

{\hbox{\upshape \tiny1\kern-2.3pt\SMALL1}}

{\hbox{\upshape \Tiny1\kern-2pt\tiny1}}}

\makeatletter

\newbox\ipbox

\newcommand{\ip}[2]{\left\langle #1\, , \,#2\right\rangle}
\newcommand{\diracb}[1]{\left\langle #1\mathrel{\mathchoice

{\setbox\ipbox=\hbox{$\displaystyle \left\langle\mathstrut
#1\right.$}

\vrule height\ht\ipbox width0.25pt depth\dp\ipbox}

{\setbox\ipbox=\hbox{$\textstyle \left\langle\mathstrut
#1\right.$}

\vrule height\ht\ipbox width0.25pt depth\dp\ipbox}

{\setbox\ipbox=\hbox{$\scriptstyle \left\langle\mathstrut
#1\right.$}

\vrule height\ht\ipbox width0.25pt depth\dp\ipbox}

{\setbox\ipbox=\hbox{$\scriptscriptstyle \left\langle\mathstrut
#1\right.$}

\vrule height\ht\ipbox width0.25pt depth\dp\ipbox}

}\right. }

\newcommand{\dirack}[1]{\left. \mathrel{\mathchoice

{\setbox\ipbox=\hbox{$\displaystyle \left.\mathstrut
#1\right\rangle$}

\vrule height\ht\ipbox width0.25pt depth\dp\ipbox}

{\setbox\ipbox=\hbox{$\textstyle \left.\mathstrut
#1\right\rangle$}

\vrule height\ht\ipbox width0.25pt depth\dp\ipbox}

{\setbox\ipbox=\hbox{$\scriptstyle \left.\mathstrut
#1\right\rangle$}

\vrule height\ht\ipbox width0.25pt depth\dp\ipbox}

{\setbox\ipbox=\hbox{$\scriptscriptstyle \left.\mathstrut
#1\right\rangle$}

\vrule height\ht\ipbox width0.25pt depth\dp\ipbox}

} #1\right\rangle}

\newcommand{\cj}[1]{\overline{#1}}

\newcommand{\bz}{\mathbb{Z}}

\newcommand{\br}{\mathbb{R}}
\newcommand{\bc}{\mathbb{C}}

\def\blfootnote{\xdef\@thefnmark{}\@footnotetext}

\renewcommand{\mod}{\operatorname{mod}}

\input xy
\xyoption{all}
\usepackage{amssymb}

\def\-{^{-1}}

\begin{document}
  \onehalfspacing

\title[Existence and exactness of exponential Riesz sequences and frames for fractal measures]{Existence and exactness of exponential Riesz sequences and frames for fractal measures}

\author{Dorin Ervin Dutkay}
\address{[Dorin Ervin Dutkay] University of Central Florida\\
	Department of Mathematics\\
	4000 Central Florida Blvd.\\
	P.O. Box 161364\\
	Orlando, FL 32816-1364\\
U.S.A.\\} \email{Dorin.Dutkay@ucf.edu}

\author{Shahram Emami}

\address{[Shahram Emami] Department of Mathematics, San Francisco State University,
1600 Holloway Avenue, San Francisco, CA 94132.}

 \email{semami1@mail.sfsu.edu,semami@hotmail.com}
\author{Chun-Kit Lai}

\address{[Chun-Kit Lai] Department of Mathematics, San Francisco State University,
1600 Holloway Avenue, San Francisco, CA 94132.}

 \email{cklai@sfsu.edu}

\thanks{}

\begin{abstract}
We study the construction of exponential frames and Riesz sequences for a class of fractal measures on $\br^d$ generated by infinite convolution of discrete measures using the idea of frame towers and Riesz-sequence towers. The exactness and overcompleteness of the constructed exponential frame or Riesz sequence is completely classified in terms of the cardinality at each level of the tower.    Using a version of the solution of the Kadison-Singer problem, known as the $R_{\epsilon}$-conjecture, we show that all these measures contain exponential Riesz sequences of infinite cardinality. Furthermore, when the measure is the middle-third Cantor measure, or more generally for self-similar measures with no-overlap condition, there are always exponential Riesz sequences of maximal possible Beurling dimension.  
\end{abstract}

\maketitle \tableofcontents

\section{Introduction}

A finite Borel  measure $\mu$ is called a {\it spectral measure} if there exists a set $\Lambda$ such that the family of exponential functions $E(\Lambda): = \{e^{2\pi i \langle\lambda, x\rangle} :\lambda\in\Lambda\}$ is an orthonormal basis for $L^2(\mu)$. A set $\Omega$ is called a {\it spectral set} if $\chi_{\Omega}dx$ is a spectral measure. If the family of exponentials $E(\Lambda)$ forms a frame/Riesz sequence for $L^2(\mu)$ (See Section 2 for the definition of frames and Riesz sequences), we say that the measure $\mu$ is frame-spectral/RS-spectral. If $E(\Lambda)$ is both a frame and a Riesz sequence, then $E(\Lambda)$ is a Riesz basis for $L^2(\mu)$ and $\mu$ is called {\it Riesz-spectral}.

\medskip

The study of spectral measures was initiated in \cite{Fug74} motivated by analysis of commuting self-adjoint extensions of partial differential operators. Fuglede asked which subsets of the Euclidean space are spectral sets and proposed his famous conjecture which states that these sets are precisely those that tile the Euclidean space by translations. In \cite{JP98}, Jorgensen and Pedersen widened the scope of Fuglede's question and asked which Borel measures on $\br^d$ admit orthogonal Fourier series. They constructed the first example of a {\it singular, non-atomic} spectral measure. It is based on a Cantor-type construction, where the unit interval is divided into four pieces and the second and fourth piece are discarded. Many more classes of examples of singular spectral measures have been constructed since, see, e.g, \cite[and the references therein]{Str00,LaWa02,DJ06,DHL18}. Strichartz proved in \cite{MR2279556} that, in some cases, the Fourier series associated to such singular spectral measures have much better convergence properties than their classical counterparts, see also \cite{MR3152727}. 

\medskip

In their original paper, Jorgensen and Pedersen also proved that the more familiar middle-third Cantor set, with the measure $\mu$ being the standard Hausdorff measure, is not a spectral measure, so it does not admit orthogonal bases of exponential functions. This motivated Strichartz \cite{Str00}, to ask if this middle-third Cantor measure $\mu$ can be frame-spectral or even Riesz-spectral.  Very little progress towards an answer for this question has been made since then. In \cite{MR2826404}, some Bessel sequences of exponential functions were constructed with positive Beurling dimension for $\mu$. In \cite{LW17}, the first examples of frame-spectral fractal measures with only finitely many mutually orthogonal exponential were constructed.

\medskip

In this paper, we generalize the study in \cite{LW17} on $\br^d$ and consider the frame-spectrality and the RS-spectrality of the  measures obtained as infinite convolutions of atomic measures,  of the form
\begin{equation}\label{measure_conv}
\mu = \mu(\{R_j,B_j\}) = \delta_{{\bf R}_1^{-1}B_1}\ast\delta_{{\bf R}_2^{-1}B_2}\ast...\ast\delta_{{\bf R}_n^{-1}B_n}\ast....,
\end{equation}
where ${\bf R}_j=R_jR_{j-1}\dots R_1$, with $R_i$ being some  expanding matrices with integer entries in $\br^d$, $B_j$ are some finite sets of digits in ${\mathbb Z}^d$, and for a finite subset $A$ of $\br^d$, 
$$\delta_A=\frac{1}{\# A}\sum_{a\in A}\delta_a,$$
where $\delta_a$ is the Dirac measure at the point $a$. This class of measures contains self-affine measures defined by affine iterated function systems as well as the middle-third Cantor measure. 


\medskip

\noindent{\bf Main Result and organization of the paper.} The main tool for constructing frames and Riesz sequences for our class of measures is based on the {\it frame/Riesz sequence towers} (Definition \ref{def2.1}). Originally, the idea of the tower construction is due to Strichartz \cite{Str00}, who considered {\it compatible towers} for constructing orthogonal exponential basis. Special cases of the tower constructions for frames were considered previously by the authors \cite{LW17,DHL18}.

\smallskip

In Section 2, we will present the most general setting for towers to generate frames and, the first examples of Riesz sequences. Basically, we will need to have a frame/Riesz sequence condition at each finite dimensional level $\{{\bf R}_j, B_j\}$ and then concatenate, or convolute these sets to obtain frames or Riesz sequences for the resulting measure. In Section 3, we show how these towers generate frames and Riesz sequences of exponential functions for the infinite convolution measure  in (\ref{measure_conv}) (Theorem \ref{thm1.3}). We notice that, similar to all previous results in literature, a tail-term estimate (See (\ref{eq1.5})) is required for the infinite convolution measure to have a frame. However, no such estimate is required for exponential Riesz sequences.

\smallskip

 In Section 4, we investigate the exactness and completeness of the resulted frame and Riesz sequences (Theorem \ref{theorem2.1} and Theorem \ref{th4.4}) generated by the frame/Riesz sequence towers. In brief, these towers show a rigid structure. Under the tail-estimate (\ref{eq1.5}), we get a Riesz basis of exponential functions if and only if we have a square matrix of finite frames at all levels. In particular, this shows that all frame-spectral measures constructed by the third-named author and Wang in \cite{LW17} are indeed Riesz-spectral.

\smallskip

The recent solution of the Kadison-Singer problem by \cite{MSS} provides an elegant proof that it is possible to partition a highly redundant tight frame into two frames with roughly the same frame bounds. This has led to important advances in frame theory. A recent survey of the Kadison-Singer problems and its equivalent statements can be found in \cite{B18}. In Section 5, using one of the consequences of the Kadison-Singer theorem, we show that all infinite convolution measures (\ref{measure_conv}) admit  Riesz sequences of exponentials of infinite cardinality (Theorem \ref{theorem_Riesz}). In the more particular case of self-similar measures, we show that there exist Riesz sequences of maximal Beurling dimension (Theorem \ref{th5.7}). 

\smallskip

Theorem \ref{th5.7} tells us that the middle-third Cantor measure has an exponential Riesz sequence of Beurling dimension $\log2/\log3.$ All spectral self-similar measures we know admit a spectrum of maximal Beurling dimension.  Our result here  leads to some evidence that an exponential  Riesz basis may exist for the middle-third Cantor measure and that would lead to a solution to Strichartz's question. On the other hand, in contrast to the existence of exponential Riesz sequences of maximal Beurling  dimension, it is also known that one can also construct exponential orthonormal basis of zero Beurling dimension for some spectral measures \cite{MR3055992}. This tells us that an exponential frame of the middle-third Cantor measure may exist even if it does not have maximal Beurling dimension. 

\smallskip


\section{Frame and Riesz sequence towers}

Recall that  a sequence of vectors $\{f_k\}_{k=1}^{\infty}$ is called a {\it frame}  for a Hilbert space $H$ if there exists $C,D>0$ such that, for all $x\in H$,
$$
C \|x\|^2 \le \sum_{k=1}^{\infty} |\langle x,f_k\rangle|^2 \le D\|x\|^2.
$$
A sequence of vectors $\{f_k\}_{k=1}^{\infty}$ is called a {\it Riesz sequence}  for a Hilbert space $H$ if there exists $C,D>0$ such that  for any finite scalar sequence $(a_k)$ (i.e., there exists $N$ such that $a_k = 0$ for all $k>N$),
$$
C \sum_{k\le N} |a_{k}|^2\le \left\|\sum_{k\le N} a_kf_k\right\|^2 \le D\sum_{k\le N} |a_{k}|^2.
$$
(see \cite{Chr03} for a comprehensive theory of frames and Riesz sequences). Our goal is to build a frame/Riesz sequence for Cantor-type fractal measures defined by rescaling. This section will be devoted to studying the finite dimensional preparation of such construction. A matrix $R$ is called {\it expanding} if all of its eigenvalues have moduli strictly greater than 1. Throughout the paper, $A^{\mathtt T}$ denote the transpose of $A$.

\subsection{Finite dimensional preliminaries}

\begin{definition}\label{def2.1}
Let $R$ be an $d\times d$ expanding matrix of integer entries and let $B,L$ be a finite subset of ${\mathbb Z}^d$ and $0\in B\cap L$ (by a simple translation, there is no loss of generality to assume this). Define the vector
$$
{\bf e}_{R,\lambda} =  \frac{1}{\sqrt{\#B}}\left( e^{2\pi i \langle R^{-1}b,\lambda\rangle}\right)_{b\in B}^{\mathtt T} \in {\mathbb C}^{\#B}
$$
We say that $(R,B,L)$ forms a {\it frame triple with bounds $C\le D$} if
$$
C\|{\bf x}\|^2 \le \sum_{\lambda\in L} |\langle {\bf x},{\bf e}_{R,\lambda}\rangle|^2\le D\|{\bf x}\|^2, \ \forall \ {\bf x}\in {\mathbb C}^{\#B}.
$$
We say that $(R,B,L)$ forms a {\it Riesz sequence triple with bounds $C\le D$} if
$$
C\sum_{\lambda\in L}|a_{\lambda}|^2 \le \left\| \sum_{\lambda\in L} a_{\lambda} {\bf e}_{R,\lambda}\right\|^2\le D\sum_{\lambda\in L}|a_{\lambda}|^2, \ \forall (a_{\lambda})\in {\mathbb C}^{\#L}.
$$

\medskip

For $j=1,2,...$, let $0<C_j\le D_j<\infty$ be a sequence of positive numbers such that $\prod_{j=1}^{\infty}C_j>0$ and $\prod_{j=1}^{\infty}D_j<\infty$. Let $R_j$ be a sequence of expanding integer matrices on ${\mathbb R}^d$ and $B_j,L_j$  are finite subsets of ${\mathbb Z}^d$. We say that $\{(R_j,B_j,L_j): j=1,2,...\}$ forms a {\it frame tower} (respectively  a {\it Riesz sequence tower}) with respect to the bounds $C_j,D_j$ if for each $j=1,2,...$, $(R_j,B_j,L_j)$ forms a frame triple (respectively a Riesz sequence triple) with bounds $C_j,D_j$.
\end{definition}

We notice that if $C = D = 1$ and $\#B = \#L$, then $\{{\bf e}_{R,\lambda}: \lambda\in\Lambda\}$ forms an orthonormal basis on ${\mathbb C}^{\#B}$. In this case $(R,B,L)$ is called the {\it Hadamard triple}, which is known to be the key condition for generating exponential orthonormal basis of fractal measures \cite{Str00,DHL18}.

\medskip

The following lemma establishes the duality relation between these two triples.
\begin{lemma}\label{lemma1}
$(R,B,L)$ forms a frame triple with bounds $C,D$ if and only if $(R^{\mathtt T}, L,B)$ forms a Riesz sequence triple with bounds $\left(\frac{\#B}{\#L}C,\frac{\#B}{\#L}D\right)$.
\end{lemma}

\begin{proof}
As $\langle  R^{-1}b,\lambda\rangle = \langle (R^{\mathtt T})^{-1}\lambda,b\rangle$, we have
$$
\sum_{\lambda\in L} |\langle {\bf x},{\bf e}_{R,\lambda}\rangle|^2 = \sum_{\lambda\in L} \left| \sum_{b\in B} x_b\frac{1}{\sqrt{\#B}} e^{-2\pi i \langle R^{-1}b,\lambda\rangle}\right|^2$$$$ = \sum_{\lambda\in L} \left| \sum_{b\in B} x_b\frac{1}{\sqrt{\#B}} e^{-2\pi i  \langle (R^{\mathtt T})^{-1}\lambda,b\rangle}\right|^2 = \frac{\#L}{\#B}\left\| \sum_{b\in B} x_b {\bf e}_{R^{\mathtt T},b}\right\|^2
$$
The lemma follows from this. 
\end{proof}

Given finite set of integers $B,L\subset{\mathbb Z}^d$ and an integral expanding matrix $R$, we define the $(\#L) \times (\#B)$ matrix
$$
{\mathcal F}_{L,B} = \frac{1}{\sqrt{\#B}} \left( e^{2\pi i \langle R^{-1}b,\lambda\rangle}\right)_{\lambda \in L, b\in B}
$$
(Rows are indexed by $L$ and columns are indexed by $B$). Then ${\bf e}_{R,\lambda}$, $\lambda\in L$ are all the row vectors in ${\mathcal F}_{L,B}$. With some simple linear algebra, we have the following proposition. Note that this proposition is well-known if $(R,B,L)$ forms a Hadamard triple \cite{LaWa02}. 

\begin{proposition}\label{Prop_distinct}
Suppose that $(R,B,L)$ forms a frame triple. Then each element in $B$ must be a  distinct representative in ${\mathbb Z}^d/R({\mathbb Z}^d)$. 
\end{proposition}

\begin{proof}
Since $\{{\bf e}_{R,\lambda}:\lambda\in\Lambda\}$ forms a frame for ${\mathbb C}^{\#B}$ if $(R,B,L)$ forms a frame triple, we have that the vectors ${\bf e}_{R,\lambda}$ span ${\mathbb C}^{\#B}$. Hence, the rank of the matrix ${\mathcal F}_{L,B}=\#B = $ number of columns. However, if there exists $b,b'$ such that $b = b'+ Rk$ for some $k\in{\mathbb Z}^d$, then, for all $\lambda\in L$, 
$$
e^{2\pi i \langle R^{-1}b,\lambda\rangle} = e^{2\pi i \langle R^{-1}(b'+Rk),\lambda\rangle} = e^{2\pi i \langle R^{-1}b',\lambda\rangle}
$$
This means that ${\mathcal F}_{L,B}$ has two identical columns. Hence, the rank of ${\mathcal F}_{L,B}$ is strictly less than $\#B$, a contradiction. Hence,  each element in $B$ must be a  distinct representative in ${\mathbb Z}^d/R({\mathbb Z}^d)$, completing the proof.
\end{proof}

Using Lemma \ref{lemma1}, it also follows easily that if $(R,B,L)$ forms a Riesz sequence triple, then each element in $L$ must be a  distinct representative in ${\mathbb Z}^d/R^{\mathtt T}({\mathbb Z}^d)$.  However, it is not true that $B$ is from distinct representative in ${\mathbb Z}^d/R({\mathbb Z}^d)$ if $(R,B,L)$ forms a Riesz sequence triple, as we will see in Example \ref{example013}. Since we wish to construct frame-spectral measures from frame towers, we will from now on assume that {\it the elements in $B$ are distinct representatives in ${\mathbb Z}^d/R({\mathbb Z}^d)$}.

\smallskip

\begin{lemma}\label{lemma_tight}
Let $R$ be an $d\times d$ integral expanding matrix and let $B$ be a set containing some  distinct representatives in ${\mathbb Z}^d/R({\mathbb Z}^d)$. Suppose that $\overline{L}$ is a complete set of distinct representatives of ${\mathbb Z}^d/R^{\mathtt T}({\mathbb Z}^d)$. Then $(R,B,\overline{L})$ forms a (tight) frame triple with constant $C = D =  \frac{|\det (R)|}{\#B}$
\end{lemma}

\begin{proof}

Given an integer expanding matrix $R$, let $\overline{B}$ be a complete set of distinct representative of the group ${\mathbb Z}^d/R({\mathbb Z}^d)$. Then it is well-known that the matrix 
$$
\frac{1}{\sqrt{|\det (R)|}} \left( e^{2\pi i \langle R^{-1}b,\lambda\rangle}\right)_{\lambda\in\overline{L},b\in\overline{B}}
$$ 
is a unitary matrix. Hence, the rows form an orthonormal basis for ${\mathbb C}^{\#\overline{B}}$. Let 
$$
\tilde{\bf e}_{R,\lambda} =  \frac{1}{\sqrt{\#\overline{ B}}}\left( e^{2\pi i \langle R^{-1}b,\lambda\rangle}\right)_{b\in \cj B}^{\mathtt T} \in {\mathbb C}^{\#\cj B}
$$
in $\bc^{\#\overline{ B}}$. Then we have 
$$
\sum_{\lambda\in\overline{L}} \left|\langle{\bf x},\tilde{\bf e}_{R,\lambda}  \rangle \right|^2 =  \frac{|\det (R)|}{\#B} \|{\bf x}\|^2.
$$
Taking ${\bf x}$ to be a vector ${\bf w}\in{\mathbb C}^{\#B}$ and zero  on entries located at  $\overline{B}\setminus B$, we have that 
$$
\sum_{\lambda\in\overline{L}} \left|\langle{\bf w},{\bf e}_{R,\lambda}  \rangle \right|^2 =  \frac{|\det (R)|}{\#B} \|{\bf w}\|^2
$$
for all ${\bf w}\in{\mathbb C}^{\#B}$. Hence, $\{{\bf e}_{R,\lambda}:\lambda\in \overline{L}\}$ forms a unit norm tight frame on ${\mathbb C}^{\#B}$ with its tight frame constant $D = \frac{|\det (R)|}{\#B}$.
\end{proof}

\subsection{Concatenation of frame/Riesz sequence triples.} Given frame/Riesz sequence towers, we can concatenate finitely many factors to form a larger frame/Riesz sequence triple. Define 
$$
{\bf R}_n = R_n...R_1
$$ and let
\begin{equation}\label{eqB_n}
 {\bf B}_n ={\bf R}_{n} \cdot \left\{\sum_{k=1}^{n}{\bf R}_k^{-1}b_k: b_k\in B_k \right\} = R_n...R_2(B_1)+R_n...R_3(B_2)+...+B_n.
\end{equation}
\begin{equation}\label{Lambda_n}
\Lambda_n = L_1+R_1^TL_2+...+(R_1^TR_2^T...R_{n-1}^T)L_n
\end{equation}
\begin{proposition}\label{proposition1.3}
With  the notations above, the following statements hold:
\begin{enumerate}
\item Suppose that $\{(R_j,B_j,L_j): j=1,2,...\}$ forms a frame tower. Then $({\bf R}_n, {\bf B}_n,\Lambda_n)$ forms a frame triple with bounds $\prod_{j=1}^nC_j$, $\prod_{j=1}^n D_j$.
\item Suppose that $(R_j,B_j,L_j)$ forms a Riesz sequence tower. Then $({\bf R}_n, {\bf B}_n,\Lambda_n)$ forms a Riesz sequence triple with bounds $\prod_{j=1}^nC_j$, $\prod_{j=1}^n D_j$. 
\end{enumerate}
\end{proposition}

\begin{proof}
(i). We prove it by mathematical induction. When $n=1$, it is the frame triple for $(R_1,B_1,L_1)$, so the statement is true trivially. Assume now the inequality is true for $n-1$. Then we decompose ${\bf b}\in {\bf B}_n$ and $\lambda\in {\Lambda}_n$ by
$$
{\bf b} = b_n+ R_n {\bf b}_{n-1}, \ \lambda = \lambda_{n-1}+ {\bf R}_{n-1}^{\mathtt T} l_n,
$$
where $b_n\in B_{n}$, ${\bf b}_{n-1}\in{\bf B}_{n-1}$, $\lambda_{n-1}\in {\Lambda}_{n-1}$ and $l_n\in L_n$. Let also ${\bf M}_n = \prod_{j=1}^{n}(\#B_j)$ we have
$$
\begin{aligned}
&\sum_{\lambda\in{\Lambda}_n}\left|\sum_{{\bf b}\in{\bf B}_{n}}w_{\bf b}\frac{1}{\sqrt{{\bf M}_n}}e^{-2\pi i \langle{\bf R}_n^{-1}{\bf b},\lambda\rangle}\right|^2\\
=&\sum_{\lambda_{n-1}\in{\Lambda}_{n-1}}\sum_{l_n\in L_n}\left|\sum_{{\bf b}_{n-1}\in{\bf B}_{n-1}}\sum_{b_n\in B_n}\frac{1}{\sqrt{{\bf M}_{n}}}w_{R_n{\bf b}_{n-1}+b_n}e^{-2\pi i \left\langle {\bf R}_{n}^{-1}(b_n+ R_n {\bf b}_{n-1}),\lambda_{n-1}+ {\bf R}_{n-1}^{\mathtt T}l_n\right\rangle}\right|^2.\\
\end{aligned}
$$
Note that $\langle{\bf R}_{n}^{-1} (R_n{\bf b}_{n-1}), {\bf R}_{n-1}^{\mathtt T}l_n\rangle = \langle{\bf b}_{n-1},l_n\rangle$ is always an integer, so the term above can be written as
$$
\sum_{\lambda_{n-1}\in{\Lambda}_{n-1}}\sum_{l_n\in L_n}\left|\sum_{b_n\in B_n}\frac{1}{\sqrt{\#B_n}}\left(\sum_{{\bf b}_{n-1}\in{\bf B}_{n-1}}\frac{1}{\sqrt{{\bf M}_{n-1}}}w_{R_n{\bf b}_{n-1}+b_n}e^{-2\pi i \langle {\bf R}_{n}^{-1}(b_n+ R_n {\bf b}_{n-1}),\lambda_{n-1}\rangle}\right)e^{-2\pi i  \langle R_n^{-1}b_n,l_n\rangle}\right|^2
$$
Using the frame triple assumption for  $(R_n,B_n, L_n)$ and also the induction hypothesis, we further get
$$
\begin{aligned}
\leq&D_n\cdot\sum_{\lambda_{n-1}\in{\Lambda}_{n-1}}\sum_{b_n\in B_n}\left|\sum_{{\bf b}_{n-1}\in{\bf B}_{n-1}}\frac{1}{\sqrt{{\bf M}_{n-1}}}w_{R_n{\bf b}_{n-1}+b_n}e^{-2\pi i \langle {\bf R}_{n}^{-1}(b_n+ R_n {\bf b}_{n-1}),\lambda_{n-1}\rangle}\right|^2\\
=&D_n \cdot \sum_{b_n\in B_{n}}\sum_{\lambda_{n-1}\in{\Lambda}_{n-1}}\left|\sum_{{\bf b}_{n-1}\in{\bf B}_{n-1}}\frac{1}{\sqrt{{\bf M}_{n-1}}}w_{R_n{\bf b}_{n-1}+b_n}e^{-2\pi i \langle {\bf R}_{n-1}^{-1}  {\bf b}_{n-1},\lambda_{n-1}\rangle}\right|^2\\
\leq &\prod_{j=1}^{n}D_j \cdot\sum_{b_n\in B_{n}}\sum_{{\bf b}_{n-1}\in{\bf B}_{n-1}}|w_{R_n{\bf b}_{n-1}+b_n}|^2\\
=&\prod_{j=1}^{n}D_j \cdot\|{\bf w}\|^2.\\
\end{aligned}
$$
This completes the proof of the upper bound and the proof of the lower bound is analogous.

\medskip

For (ii), by Lemma \ref{lemma1}, we note that $(R_j^{\mathtt T}, L_j,B_j)$ with $j=n,n-1,...,1$ (in reverse order) now forms a frame tower with frame bound $\left(\frac{\#B_j}{\#L_j}C,\frac{\#B_j}{\#L_j}D_j\right)$. Then we know that 
$(R_1^{\mathtt T}...R_n^{\mathtt T}, {\bf L}_n, \widetilde{\bf B}_n)$ forms a frame triple with bounds $\left(\prod_{j=1}^n\frac{\#B_j}{\#L_j}C_j,\prod_{j=1}^n\frac{\#B_j}{\#L_j}D_j\right)$, where 
$$
{\bf L_n} =  R^{\mathtt T}_1...R^{\mathtt T}_{n-1}(L_n)+R^{\mathtt T}_1...R^{\mathtt T}_{n-2}(L_{n-1})+...+L_1 = \Lambda_n,
$$
by replacing $R_k$ with $R_{n-k+1}^{\mathtt T}$ and $B_k$ with $L_{n-k+1}$ in (\ref{eqB_n}), and similarly
$$
\widetilde{{\bf B}}_n = B_n+...+R_n...R_2(B_1) = {\bf B}_n.
$$

\medskip

Hence, $(R_1^{\mathtt T}...R_n^{\mathtt T}, \Lambda_n, {\bf B}_n)$ forms a frame triple with bounds $\left(\prod_{j=1}^n\frac{\#B_j}{\#L_j}C_j,\prod_{j=1}^n\frac{\#B_j}{\#L_j}D_j\right)$. Using Lemma \ref{lemma1} again, $({\bf R}_n, {\bf B}_n, \Lambda_n)$ forms a Riesz sequence triple.
\end{proof}

\section{Frame-spectral/RS-spectral Cantor measures }

In this section, we will use the frame/Riesz-sequence towers to generate Cantor measures with Fourier frames and Riesz sequences. Given a sequence of expanding matrices $R_n$ with integer entries  and a finite collection of integer digit sets  $B_j$, a natural probability measure is induced 
\begin{equation}\label{mu}
\mu = \mu(\{R_j,B_j\}) = \delta_{{\bf R}_1^{-1}B_1}\ast\delta_{{\bf R}_2^{-1}B_2}\ast...\ast\delta_{{\bf R}_n^{-1}B_n}\ast....,
\end{equation}
and we assume that the infinite convolution product is weakly convergent to a Borel probability measure.  We let
  $$
  \mu_n = \delta_{{\bf R}_1^{-1}B_1}\ast\delta_{{\bf R}_2^{-1}B_2}\ast...\ast\delta_{{\bf R}_n^{-1}B_n}, \ \mu_{>n} = \delta_{{\bf R}_{n+1}^{-1}B_{n+1}}\ast\delta_{{\bf R}_{n+2}^{-1}B_{n+2}}\ast...
  $$
so that $\mu = \mu_n\ast\mu_{>n}$.  Let also
  $$
 K_{n} = \left\{\sum_{k=n+1}^{\infty}{\bf R}_k^{-1}b_k: b_k\in B_k \right\},  {\bf B}_n = {\bf R}_n\left\{\sum_{k=1}^{n}{\bf R}_k^{-1}b_k: b_k\in B_k \right\}.
  $$
 Hence,  $K_0 = \bigcup_{{\bf b}\in {\bf B}_n}({\bf R}_n^{-1}{\bf b}_n+K_n)$ and $K_0, {\bf R}_n^{-1}{\bf B}_n, K_n$ are respectively the support of $\mu,\mu_n$ and $\mu_{>n}.$ We say that $\mu$ satisfies the {\it no overlap condition} if
 $$
 \mu (({\bf R}_n^{-1}{\bf b}_n+ K_n)\cap({\bf R}_n^{-1}{\bf b}_n'+K_n))=0, \ \mbox{for all} \ {\bf b}_n\neq{\bf b}_n'\in{\bf B}_n, \ \mbox{for all}  \ n\in{\mathbb N}.
 $$
Let also $K_{{\bf b},n} = {\bf R}_n^{-1}{\bf b}+ K_n$ if ${\bf b}\in {\bf B}_n$.

Recall that the Fourier transform of a finite Borel measure $\mu$ on $\br^d$ is defined as 
$$\widehat\mu(y)=\int e^{-2\pi i\ip{y}{x}}\,d\mu(x),\quad(x\in\br^d).$$
\begin{remark}
{\it We will assume throughout the paper that the no-overlap condition holds}. If ${\bf R}_j = N_j$ are integers in dimension one and $B_j$ are chosen from $\{0,1,...,N_j-1\}$, then $\mu$ are the Moran-type measure studied in \cite{LW17}. On the other hand, if all $R_j $ are the same matrix and $B_j$ is a subset of distinct representatives in the group ${\mathbb Z}^d/R({\mathbb Z}^d)$, then the resulting measure is a self-affine measure. They all satisfy the no-overlap condition. For the latter case, the no-overlap condition was proved in \cite[Section 2]{DHL18}.
\end{remark}

\begin{lemma}\label{lem3.2}
Under the no-overlap condition for the measure $\mu$, suppose that $f=\sum_{{\bf b}\in{\bf B}_n}w_{\bf b}{\bf 1}_{K_{{\bf b},n}}$. Then, 
\begin{equation}\label{eq_FT_mu}
\int f(x)e^{-2\pi i\ip{\lambda}{x}}\,d\mu(x)=\frac{1}{\#{\bf B}_n}\widehat{\mu_{>n}}(\lambda)\sum_{b\in{\bf B}_n}w_b e^{-2\pi i\ip{{\bf R}_n^{-1}b}{\lambda}}.
\end{equation}
\begin{equation}\label{eq_norm}
\int|f|^2d\mu = \frac{1}{\#{\bf B}_n} \sum_{{\bf b}\in{\bf B}_n} |w_{\bf b}|^2.
\end{equation}
\end{lemma}

\begin{proof}
We have 
$$\int f(x)e^{-2\pi i\ip{\lambda}{x}}\,d\mu(x)=\sum_{{\bf b}\in{\bf B}_n}w_{\bf b}\int {\bf 1}_{K_{{\bf b},n}}(x)e^{-2\pi i \ip{\lambda}{x}}\,d(\mu_n*\mu_{>n})(x)
$$$$=\sum_{{\bf b}\in{\bf B}_n}w_b\int {\bf 1}_{{\bf R}_n^{-1}{\bf b}+K_{0,n}}(x+y)e^{-2\pi i \ip{\lambda}{x+y}}\,d\mu_n(x)\,d\mu_{>n}(y).$$
Note that $\mu_{>n}$ is supported on $K_{0,n}$ and $\mu_n$ is supported on ${\bf R}_n^{-1}{\bf B}_n$, and, due to the non-overlap condition, $x$ has to be equal to ${\bf R}_n^{-1}b$ to get non-zero contribution. Thus, the quantity above is equal to 
$$=\sum_{{\bf b}\in{\bf B}_n}\frac{1}{\#{\bf B}_n}\int {\bf 1}_{{\bf R}_n^{-1}{\bf b}+K_{0,n}}({\bf R}_n^{-1}{\bf b}+y)e^{-2\pi i \ip{\lambda}{{\bf R}_n^{-1}{\bf b}+y}}\,d\mu_{>n}(y)$$
$$=\sum_{{\bf b}\in{\bf B}_n}\frac{1}{\#{\bf B}_n}e^{-2\pi i \ip{\lambda}{{\bf R}_n^{-1}{\bf b}}}\int e^{-2\pi i\ip{\lambda}{y}}\,d\mu_{>n}(y)$$
$$=\frac{1}{\#{\bf B}_n}\widehat{\mu_{>n}}(\lambda)\sum_{{\bf b}\in{\bf B}_n}w_{\bf b} e^{-2\pi i\ip{{\bf R}_n^{-1}{\bf b}}{\lambda}}.$$
(\ref{eq_norm}) follows from a standard computation.
\end{proof}


%
%
%

Define $\Lambda_n$ as in \eqref{Lambda_n} and let 
\begin{equation}
\Lambda=\bigcup_{n=1}^\infty\Lambda_n.
\label{eqLambda}
\end{equation}

\begin{theorem}\label{thm1.3}
\begin{enumerate}
\item Suppose that $(R_j,B_j,L_j)$ forms a frame tower with the associated measure $\mu = \mu(R_j,B_j)$ in (\ref{mu}) and that the no-overlap condition is satisfied. Suppose furthermore that 
\begin{equation}\label{eq1.5}
\delta(\Lambda) = \inf_{n\ge1}\inf_{\lambda\in\Lambda_n} |\widehat{\mu_{>n}}(\lambda)|^2 >0
\end{equation}
 Then $\{e^{2\pi i \langle \lambda,x\rangle}: \lambda\in\Lambda\}$ forms a frame for $L^2(\mu)$ with bounds $ \prod_{j=1}^{\infty}C_j$, $\prod_{j=1}^{\infty} D_j$.

\smallskip

\item Suppose that $(R_j,B_j,L_j)$ forms a Riesz sequence tower and that the associated measure satisfies the no-overlap condition. 
 Then $\{e^{2\pi i \langle \lambda,x\rangle}: \lambda\in\Lambda\}$ forms a Riesz sequence for $L^2(\mu)$ with bounds $\prod_{j=1}^{\infty}C_j$, $\prod_{j=1}^{\infty} D_j$. 
 
\end{enumerate}
\end{theorem}

\begin{proof}
Let ${\mathcal S}_n= \{\sum_{{\bf b}\in {\bf B}_n}w_{\bf b} {\bf 1}_{K_{{\bf b},n}}: w_{\bf b}\in {\mathbb C}\}$ and ${\mathcal S} = \bigcup_{n=1}^{\infty} {\mathcal S}_n$. 
To prove (i), we notice that it suffices to check the frame inequality holds for every function $f\in {\mathcal S}$ as they are dense in $L^2(\mu)$. Since $0\in B_n$ for all $n$, the collection ${\mathcal S}_n$ is an increasing union. Given any $f\in {\mathcal S}_{n_0}$ and write it as $f = \sum_{{\bf b}\in {\bf B}_{n_0}}w_{\bf b} {\bf 1}_{K_{{\bf b},{n_0}}}$, for any $n\ge n_0$, using (\ref{eq_FT_mu}) in Lemma \ref{lem3.2} (as the no-overlap condition is satisfied),
$$
\sum_{\lambda\in \Lambda_n} \left|\int f(x)e^{-2\pi i \langle\lambda,x\rangle}d\mu(x)\right|^2 = \frac{1}{\#{\bf B}_n}\sum_{\lambda\in\Lambda_n} |\widehat{\mu_{>n}}(\lambda)|^2\left|\sum_{b\in{\bf B}_n}w_b e^{-2\pi i\ip{{\bf R}_n^{-1}b}{\lambda}} \right|^2.
$$
With the assumption that $\delta(\Lambda)>0$, using that $({\bf R}_n, {\bf B}_n, \Lambda)$ forms a frame triple (Proposition \ref{proposition1.3}) and with (\ref{eq_norm}) in Lemma \ref{lem3.2}, we have 
$$
\delta(\Lambda)\cdot\left(\prod_{j=1}^n C_j \right)\cdot \int|f|^2d\mu \le \sum_{\lambda\in \Lambda_n} \left|\int f(x)e^{-2\pi i \langle\lambda,x\rangle}d\mu(x)\right|^2\le \left(\prod_{j=1}^n D_j \right)\cdot \int|f|^2d\mu
$$
Taking $n\rightarrow \infty$, we show that $\{e^{2\pi i \langle \lambda,x\rangle}: \lambda\in\Lambda\}$ forms a frame for $L^2(\mu)$ with a less sharp frame bound $\delta(\Lambda)\cdot\left(\prod_{j=1}^\infty C_j \right)$.

We now adopt the idea of the proof from  \cite{MR3055992} and \cite[Appendix]{DHL18} to show that $\delta(\Lambda)$ does not appear in the lower frame bound. It suffices to show that the lower bound of  frame inequality holds for a dense set of functions in $L^2(\mu)$. We will check it for step functions in ${\mathcal S}$. Let $f = \sum_{{\bf b}\in {\bf B}_n}w_{\bf b}{\bf 1}_{K_{\bf b},n}\in {\mathcal S}_n$ and note that $f\in{\mathcal S}_m$, for all $m\ge n$; we use this to define the coefficients $w_{\bf b}$, for $b\in {\bf B}_m$, so $f$ can be written also as a function in $\mathcal S_m$ as 
$$
f = \sum_{{\bf b}\in {\bf B}_m}w_{\bf b}{\bf 1}_{K_{\bf b},m}.$$
 Define
$$
Q_{\infty} (f) = \sum_{\lambda\in\Lambda}\left|\int f(x)e^{-2\pi i \lambda x}d\mu(x)\right|^2  =\lim_{n\rightarrow\infty}Q_n(f),
$$
where, with Lemma \ref{lem3.2},
$$
Q_n(f): = \sum_{\lambda\in \Lambda_n}\left|\int f(x)e^{-2\pi i \lambda x}d\mu(x)\right|^2 = \frac{1}{{\bf M}_n}\sum_{\lambda\in\Lambda_n}|\widehat{\mu_{>n}}(\lambda)|^2\left|\sum_{{\bf b}\in B_n}w_{\bf b} \frac{1}{\sqrt{{\bf M}_n}}e^{-2\pi i  \langle {\bf R}_n^{-1}{\bf b}, \lambda\rangle}\right|^2.
$$
 Let ${\mathbf C}_n = \prod_{j=1}^{n}C_j$, ${\mathbf D}_n = \prod_{j=1}^{n} D_j$ for $n=1,2...$ and $n=\infty$ and ${\mathbf M}_n = \prod_{j=1}^n (\#B_j)$. 
We are going to establish the lower bound. Note that 
$$
\begin{aligned}
Q_m(f) =& Q_n(f)+ \sum_{\lambda\in\Lambda_m\setminus \Lambda_{n}}\left|\int f(x)e^{-2\pi i \langle\lambda, x\rangle}d\mu(x)\right|^2\\
=&Q_n(f)+\sum_{\lambda\in\Lambda_m\setminus \Lambda_{n}}\frac{1}{{\mathbf M}_m}|\widehat{\mu_{>m}}(\lambda)|^2\left|\sum_{{\bf b}\in {\bf B}_m}w_{\bf b} \frac{1}{\sqrt{\bf M}_m}e^{-2\pi i  \langle {\bf R}_m^{-1}{\bf b}, \lambda\rangle}\right|^2.\\
\ge& Q_n(f)+\delta(\Lambda)\cdot\sum_{\lambda\in\Lambda_m\setminus \Lambda_{n}}\frac{1}{{\mathbf M}_m}\left|\sum_{{\bf b}\in {\bf B}_m}w_{\bf b} \frac{1}{\sqrt{\bf M}_m}e^{-2\pi i  \langle {\bf R}_m^{-1}{\bf b}, \lambda\rangle}\right|^2.\\
\end{aligned}
$$
Note that, by Proposition \ref{proposition1.3}(i) and Lemma \ref{lem3.2}
$$
\sum_{\lambda\in \Lambda_m}\frac{1}{{\bf M}_m}\left|\sum_{{\bf b}\in B_m}w_{\bf b} \frac{1}{\sqrt{{\bf M}_m}}e^{-2\pi i  \langle {\bf R}_m^{-1}{\bf b}, \lambda\rangle}\right|^2\ge {\bf C}_m\cdot\frac{1}{{\bf M}_m} \sum_{{\bf b}\in{\bf B}_m} |w_{\bf b}|^2  ={\bf C}_m\int|f|^2d\mu.
$$
We further have
$$
\begin{aligned}
Q_m(f) \ge& Q_n(f)+\delta(\Lambda)\cdot \left({\bf C}_m\int|f|^2d\mu-\sum_{\lambda\in \Lambda_n}\frac{1}{{\bf M}_m}\left|\sum_{{\bf b}\in B_m}w_{\bf b} \frac{1}{\sqrt{{\bf M}_m}}e^{-2\pi i  \langle {\bf R}_m^{-1}{\bf b}, \lambda\rangle}\right|^2\right)\\
= &Q_n(f)+\delta(\Lambda)\cdot \left({\bf C}_m\int|f|^2d\mu-\sum_{\lambda\in \Lambda_n}\left|\int f(x)e^{-2\pi i \lambda x}d\mu_m(x)\right|^2\right)
\end{aligned}
$$
For a fixed $n$, we let $m$ go to infinity. By the fact that $Q_{m}(f)$ converges to $Q_{\infty}(f)$ and $\mu_m$ converges weakly to $\mu$, we have
$$
Q_{\infty}(f)\ge Q_n(f) +\delta(\Lambda)\cdot \left({\bf C}_{\infty}\int|f|^2d\mu- \sum_{\lambda\in \Lambda_n}\left|\int f(x)e^{-2 \pi i \lambda x}d\mu(x)\right|^2\right).
$$
We then let $n$ go to infinity and obtain
$$
Q_{\infty}(f)\ge Q_{\infty}(f) +\delta(\Lambda)\cdot \left({\bf C}_{\infty}\int|f|^2d\mu-\sum_{\lambda\in \Lambda}\left|\int f(x)e^{-2\pi i \lambda x}d\mu(x)\right|^2\right).
$$
and thus
$$
\delta(\Lambda)\cdot \left({\bf C}_{\infty}\int|f|^2d\mu-\sum_{\lambda\in \Lambda}\left|\int f(x)e^{-2\pi i \lambda x}d\mu(x)\right|^2\right)\leq 0.
$$
However, $\delta(\Lambda)>0$ and we have
$$
{\bf C}_{\infty}\int|f|^2d\mu\le \sum_{\lambda\in \Lambda}\left|\int f(x)e^{-2\pi i \lambda x}d\mu(x)\right|^2
$$
This establishes the lower bound.

\bigskip

We now prove (ii). Take any finite subset $\Lambda_0$ of ${\Lambda}$. As $\Lambda_n$ is an increasing union, we have that $\Lambda_0\subset\Lambda_n$ for $n$ large. Note that 
$$\left\| \sum_{\lambda\in\Lambda_0}a_{\lambda}e^{2\pi i \langle\lambda,x\rangle}\right\|_{L^2(\mu_n)}^2= \frac{1}{\#{\bf B}_n}\sum_{{\bf b}\in {\bf B}_n} \left| \sum_{\lambda\in \Lambda_0} a_{\lambda} e^{2\pi i \langle {\bf R}_n^{-1}b,\lambda\rangle} \right|^2 =\left\| \sum_{\lambda\in \Lambda_0} a_{\lambda} {\bf e}_{R,\lambda}\right\|^2.
$$
 Hence, by Proposition \ref{proposition1.3},
\begin{equation}\label{eq1.2}
\left(\prod_{j=1}^nC_j\right) \cdot \sum_{\lambda\in\Lambda_0}|a_{\lambda}|^2 \le \left\| \sum_{\lambda\in\Lambda_0}a_{\lambda}e^{2\pi i \langle\lambda,x\rangle}\right\|_{L^2(\mu_n)}^2\le \left(\prod_{j=1}^nD_j\right) \cdot \sum_{\lambda\in\Lambda_0}|a_{\lambda}|^2
\end{equation}
As $\mu_n$ converges weakly to $\mu$ and $\sum_{\lambda\in\Lambda_0}a_{\lambda}e^{2\pi i \langle\lambda,x\rangle}$ is a continuous function, 
$$
\lim_{n\rightarrow\infty}\left\| \sum_{\lambda\in\Lambda_0}a_{\lambda}e^{2\pi i \langle\lambda,x\rangle}\right\|_{L^2(\mu_n)}^2 = \left\| \sum_{\lambda\in\Lambda_0}a_{\lambda}e^{2\pi i \langle\lambda,x\rangle}\right\|^2_{L^2(\mu)}.
$$
Hence, our conclusion follows by taking limit in (\ref{eq1.2}).
\end{proof}

\medskip

\begin{remark}
\begin{enumerate}
\item The inequality $\delta(\Lambda)>0$ may be regarded as a condition guaranteeing that the tail-term of $\widehat{\mu}$ does not become too small. It is a sufficient condition for the canonical spectrum $\Lambda$ to be a frame-spectrum. But it is in general not necessary (See \cite{MR3055992}). 
\item Theorem \ref{thm1.3} (ii) shows that there is no extra condition for $\Lambda$ to be a Riesz sequence once we have formed our Riesz sequence tower. However, it may happen that all the sets $L_j$ have only one element, then trivially, the set $\Lambda$ has only one element which must form a Riesz sequence. Hence, to construct an infinite Riesz sequence, we need to make sure $\#L_j\ge 2$. We will show, using a version of the Kadison-Singer theorem, that such a Riesz sequence always exists (Section \ref{Section_KS}). 
\end{enumerate}
\medskip





\end{remark}

\section{Exactness and Overcompleteness}

In this section, we will study the exactness and overcompleteness of the Fourier frame generated by the frame and Riesz-sequence towers. 
\subsection{Exactness and overcompleteness of the frame tower}
\begin{theorem}\label{theorem2.1}
Let $\{(R_j, B_j, L_j): j\ge 1\}$ be a frame tower with no-overlap. Suppose that $\delta(\Lambda)>0$ (see \eqref{eq1.5}). Then

\medskip

(i) Suppose that all $\#B_j = \#L_j$. Then $E(\Lambda) = \{e^{2\pi i \langle\lambda,x\rangle}:\lambda\in\Lambda\}$ is a Riesz basis for $L^2(\mu)$.

\medskip

(ii) Suppose that there exists $j$ such that $\#B_j<\#L_j$. Then $\{e^{2\pi i \langle\lambda,x\rangle}:\lambda\in\Lambda\}$ is a Fourier frame for $L^2(\mu)$ with infinite redundancies (i.e., there exists an infinite subset $\Lambda_0$ of $\Lambda$ such that $\Lambda\setminus\Lambda_0$ is still a frame spectrum for $L^2(\mu)$).
\end{theorem}

\begin{proof}
\noindent (i) Since $\delta(\Lambda)>0$, we know that $E(\Lambda)$ is a frame. We now show that it is a Riesz basis by showing that $E(\Lambda)$ is an exact frame. Let $\lambda_0$ be an element in $\Lambda$ and we need to show that $E(\Lambda\setminus\{\lambda_0\})$ is incomplete. 

\medskip

To show this, we note that  there exists $n_0$ such that $\lambda_0\in\Lambda_n$ for all $n\ge n_0$. From the assumption, $\#\Lambda_{n}= \#{\bf B}_n = \prod_{j=1}^n\#B_j$. Since $L^2(\mu_n)$ is finite dimensional with dimension $\#{\bf B}_n$, we have that $E(\Lambda_n)$ forms a Riesz basis for $L^2(\mu_n)$.  In particular, there exists $f_n\in L^2(\mu_n)$ such that $\|f_n\|_{L^2(\mu_n)} = 1$ and
\begin{equation}\label{eq2.1}
\frac{1}{\sqrt{\#{\bf B}_n}}\langle {\bf w}^n, {\bf e}_{{\bf R}_n,\lambda}\rangle_{{\mathbb C}^{\#{\bf B}_n}} = \langle f_n, e^{2\pi i \langle \lambda, x\rangle}\rangle_{L^2(\mu_n)} = 0, \ \forall \lambda\in \Lambda_n\setminus\{\lambda_0\}.
\end{equation}
where $w^n_b = f_n(b)$ and ${\bf w}^n = (w^n_b)_{b\in{\bf B}_n}$. As $\|f_n\|_{L^2(\mu_n)} = 1$, $\|{\bf w}^n\|^2 = \#{\bf B}_n$.

\medskip

The vector $f_n$ can be identified naturally with $f_n = \sum_{b\in{\bf B}_n} w^n_b {\bf 1}_{K_{b,n}}\in L^2(\mu)$.
$$
\|f_n\|_{L^2(\mu)}^2 = \frac{1}{\#{\bf B}_n}\sum_{b\in {\bf B}_n} |w_{b,n}|^2=\|f_n\|^2_{L^2(\mu_n)}=1
$$
By the Banach-Alaoglu theorem, there is a subsequence $f_{n_k}$ that converges weakly to some function $f\in L^2(\mu)$. Note that for all $\lambda\in\Lambda\setminus\{\lambda_0\}$, $\lambda\in \Lambda_{n_k}$ for all $k$ sufficiently large, and we have, with Lemma \ref{lem3.2},
\begin{equation}\label{eq2.2}
\langle f_{n_k}, e^{2\pi i \langle \lambda, x\rangle}\rangle_{L^2(\mu)} = \widehat{\mu_{>n_k}}(\lambda) \langle f_{n_k}, e^{2\pi i \langle\lambda,x\rangle}\rangle_{L^2(\mu_{n_k})} = 0.
\end{equation}
By taking $k\rightarrow\infty$, we obtain that 
$$
\langle f, e^{2\pi i \langle \lambda, x\rangle}\rangle_{L^2(\mu)} = 0, \ \forall \lambda\in\Lambda\setminus\{\lambda_0\}.
$$
Therefore, $f$ is orthogonal to all $e^{2\pi i \langle\lambda,x\rangle}$, $\lambda\in\Lambda\setminus\{\lambda_0\}$. We now show that $f\ne 0$ in $L^2(\mu)$, then this implies that $E(\Lambda\setminus\{\lambda_0\})$ is not complete. 

\medskip

Indeed, if we take ${\bf w}^n$ into the definition of the frame triple of $({\bf R}_n,{\bf B}_n,\Lambda_n)$ and use (\ref{eq2.1}), we obtain
$$
\left(\prod_{j=1}^nC_j\right)\|{\bf w}^n\|^2 \le \sum_{\lambda\in\Lambda_n} |\langle {\bf w}^n,{\bf e}_{{\bf R}_n,\lambda}\rangle|^2 = |\langle {\bf w}^n,{\bf e}_{{\bf R}_n,\lambda_0}\rangle|^2 = \#{\bf B}_n  |\langle f_n, e^{2\pi i \langle \lambda_0, x\rangle}\rangle_{L^2(\mu_n)}|^2.
$$
Note that $ \|{\bf w}^n\|^2=\#{\bf B}_n$ and  $|\widehat{\mu_{>n}}(\lambda)|^2\ge \delta(\Lambda)>0$. By (\ref{eq2.2}), we have
$$
|\langle f_n, e^{2\pi i \langle \lambda_0, x\rangle}\rangle_{L^2(\mu)}|^2 \ge \left(\prod_{j=1}^nC_j\right) \cdot \delta(\Lambda).
$$
Take $n=n_k$ to infinity, we obtain
$$
|(fd\mu)^{\widehat{}}(\lambda_0)|^2 \ge \left(\prod_{j=1}^{\infty}C_j\right) \cdot \delta(\Lambda)>0.
$$
Since $(fd\mu)^{\widehat{}}(\lambda_0)\ne 0$, $f$ cannot be a zero function in $L^2(\mu)$ and this completes the proof.

\bigskip

(ii) Suppose that at the $j_0$th level, we have that  $\#B_{j_0}<\#L_{j_0}$. Then the collection of vectors $\{{\bf e}_{R_{j_0},\lambda}: \lambda\in L_{j_0}\}$ is a linearly dependent set in ${\mathbb C}^{\#B_{j_0}}$. Since we have a finite dimensional vector space,  there exists $\lambda_{j_0}\in L_{j_0}$ such that $\{{\bf e}_{R_{j_0},\lambda}: \lambda\in L_{j_0}\setminus\{\lambda_{j_0}\}\}$ is still a frame for ${\mathbb C}^{\#B_{j_0}}$ with frame bounds $\widetilde{C_j}$ and $\widetilde{D_j}$. Let $\widetilde{L_{j_0}} = L_{j_0}\setminus\{\lambda_{j_0}\}$. 

\medskip

Then it is clear that $\left(\{(R_j,B_j,L_j): j\ge 1\}\setminus \{(R_{j_0},B_{j_0},L_{j_0})\} \right)\cup \{(R_{j_0},B_{j_0},\widetilde{L_{j_0}})\}$ still forms a frame tower with bounds
$$
\frac{\widetilde{C_j}}{C_j}\cdot\prod_{j=1}^{\infty}C_j \ \mbox{and} \   \frac{\widetilde{D_j}}{D_j}\cdot\prod_{j=1}^{\infty}D_j
$$
which are still positive and finite. Moreover, the corresponding $
\widetilde{\Lambda}$ formed with $L_{j_0}$ replaced by $\widetilde{L_{j_0}}$ in (\ref{Lambda_n})   is a subset of $\Lambda$. This implies that  $\delta(\widetilde{\Lambda})\ge \delta(\Lambda)>0$. By Theorem \ref{thm1.3}(i), 
$E(\widetilde{\Lambda})$ is a frame for $L^2(\mu)$. Note that the removed elements are
$$
\Lambda\setminus\widetilde{\Lambda} = R_1^{\mathtt T}..R_{j_0-1}^{\mathtt T}\lambda_{j_0}+ \left\{ \sum_{finite,j\ne j_0} (R_1^{\mathtt T}...R_{j-1}^{\mathtt T})\ell_{j}: \ell_j\in L_j\right\},
$$
which is an infinite set. Hence, $E(\Lambda)$ is a Fourier frame for $L^2(\mu)$ with infinite redundancies.
\end{proof}

\medskip

\begin{remark}
In \cite{LW17}, we considered on ${\mathbb R}^1$, $R_j =  M_jK_j+\alpha_j$, where $M_j,K_j$ are integers and $0\le \alpha_j<M_j$ and they satisfy
$$
\sum_{j=1}^{\infty} \frac{\alpha_j\sqrt{M_j}}{K_j}<\infty.
$$
Then letting $B_j = \{0,K_j,...,(M_j-1)K_j\}$ and $L_j = \{0,1,...,M_j-1\}$, the resulting $(R_j,B_j,L_j)$ forms a frame tower and the associated fractal measure admits a Fourier frame. Using Theorem \ref{theorem2.1}, we can actually conclude that the Fourier frame we constructed is actually a Riesz basis. 
\end{remark}

\medskip

\subsection{Incompleteness of Riesz sequence tower}
In this section we consider the completeness of the Riesz basis obtained from a Riesz-sequence tower.  We first recall a proposition about Riesz sequences, whose statements can be found in Young's book (\cite{Young}, Proposition 2 and Theorem 3 in Chapter 4, p.129).

\begin{proposition}\label{proposition_Young}
Let $H$ be a Hilbert space and let $\{f_n: n\ge 1\}$ be a Riesz sequence for $H$. Let  $C$ be its lower bound. Then for any $\ell^2$ sequence $\{c_n\}$, there exists $f\in H$ such that 
$$
\|f\|\le \frac{1}{C} \sum_{n=1}^{\infty}|c_n|^2
$$ 
and
$$
\langle f,f_n\rangle =c_n, \forall n\ge1.
$$
\end{proposition}

\medskip

\begin{theorem}\label{th4.4}
Let $\{(R_j, B_j, L_j): j\ge 1\}$ be a Riesz sequence tower and assume the associated fractal measure satisfies the no-overlap condition and elements in $B_j$ are  distinct representatives in ${\mathbb Z}^d/R({\mathbb Z}^d)$. Suppose that there exists $j$ such that $\#L_j<\#B_j$. Then $\{e^{2\pi i \langle\lambda,x\rangle}:\lambda\in\Lambda\}$ is not complete in  $L^2(\mu)$. 
\end{theorem}

\begin{proof}
Without loss of generality, we may assume that, at the first level, we have $\#L_1<\#B_1$. From Proposition \ref{Prop_distinct}, elements in $L_1$   are distinct representative in ${\mathbb Z}^d/R^{\mathtt T}({\mathbb Z}^d)$.  Let $\overline{L}_1$ be a complete representative of  ${\mathbb Z}^d/R^{\mathtt T}({\mathbb Z}^d)$ containing $L_1$. As elements in $B_1$ are distinct representative in ${\mathbb Z}^d/R({\mathbb Z}^d)$,  $\{{\bf e}_{R,\lambda}: \lambda\in \overline{L}_1\}$ forms a tight frame in ${\mathbb C}^{\#B}$ by Lemma \ref{lemma_tight}. On the other hand, as $\{{\bf e}_{R,\lambda}: \lambda\in L_1\}$ forms a Riesz sequence, so they must be linearly independent. Hence,  we can find $\lambda_1\in \overline{L_1}$ such that $(R_1,B_1,\widetilde{L_1})$ with $\widetilde{L_1} = L_1\cup\{\lambda_1\}$ forms a Riesz sequence triple and 
$(R_1,B_1,\widetilde{L_1}) \cup \{(R_j, B_j, L_j): j\ge 2\}$ forms a Riesz sequence tower. 

\medskip

We know that $({\bf R}_n, {\bf B}_n, \widetilde{\Lambda_n})$ are Riesz sequence triples for all $n\ge1$, and 
$$
\widetilde{\Lambda_n} = \widetilde{L_1}+ R_1^{\mathtt T} L_2+...+R_1^{\mathtt T}...R_{n-1}^{\mathtt T} L_n.
$$ 
Define $\widetilde{\Lambda}$ analogously. Then $\Lambda_n\subset\widetilde{\Lambda_n}$ and $\Lambda\subset\widetilde{\Lambda}$. 
Let 
$$
c_{\lambda_1} = 1 \  \mbox{and} \ c_{\lambda} = 0  , \mbox{ for all } \ \lambda\in\widetilde{\Lambda_n}\setminus\{\lambda_1\}.
$$
Then $\sum_{\lambda\in \widetilde{\Lambda_n}}|c_{\lambda}|^2 = 1$.  By Proposition \ref{proposition1.3} and \ref{proposition_Young}, we can find $f_n\in L^2(\mu_n)$ such that 
$$
\|f_n\|_{L^2(\mu_n)} \le \frac{1}{\prod_{j=1}^{n}C_j}, \ \mbox{and}  \ \langle f_n\, ,\, e^{2\pi i \langle\lambda, x\rangle}\rangle = c_{\lambda}
$$
for all $\lambda\in\Lambda_n$. Since $0<\prod_{j=1}^{\infty}C_j<\infty$ and all $C_j>0$, $C: = \inf\{\prod_{j=1}^nC_j: n\ge 1\}>0$. Identifying $f_n$ naturally in $L^2(\mu)$ with a step function, we have that 
$$
\|f_n\|_{L^2(\mu)}\le \frac{1}{ C}.
$$ 
By the Banach-Alaoglu theorem, we have a subsequence $f_{n_k}$ that converges weakly to some function $f\in L^2(\mu)$. Then, with Lemma \ref{lem3.2},
$$
\langle f_{n_k}, e^{2\pi i \langle \lambda,x\rangle}\rangle_{L^2(\mu)} = \widehat{\mu_{>n_k}}(\lambda) \cdot \langle f_{n_k}, e^{2\pi i \langle \lambda,x\rangle}\rangle_{L^2(\mu_{n_k})} .
$$
 In particular, by taking limit and since $\Lambda \subset\widetilde{\Lambda}$, we have
$$
\langle f, e^{2\pi i \langle \lambda,x\rangle}\rangle_{L^2(\mu)} =0,   \mbox{ for all } \lambda\in \Lambda.
$$
If we can show that $f\ne 0$, then $f$ is orthogonal to the the closure of the span of $e^{2\pi i \langle\lambda,x\rangle}$, $\lambda\in\Lambda$. This will show that $E(\Lambda)$ is not complete. 

\medskip

To show that $f\ne 0$, we note that by our choice of $c_{\lambda},$
$$
\langle f_n, e^{2\pi i \langle\lambda_1,x\rangle}\rangle_{L^2(\mu_n)} = 1.
$$
This implies that 
$$
\langle f_n, e^{2\pi i \langle\lambda_1,x\rangle}\rangle_{L^2(\mu)} = \widehat{\mu_{>n}}(\lambda_1)\langle f_n, e^{2\pi i \langle\lambda_1,x\rangle}\rangle_{L^2(\mu_n)} = \widehat{\mu_{>n}}(\lambda_1).
$$
We note that the measure $\mu_{>n}$ converges weakly to $\delta_0$. This means that $\widehat{\mu_{>n}}(\cdot)$ converges to $1$ uniformly on all compact subsets of ${\mathbb R}^d$. Hence, $\widehat{\mu_{>n}}(\lambda_1)$ converges to $1$ and we have 
$$
\langle f, e^{2\pi i \langle\lambda_1,x\rangle}\rangle_{L^2(\mu)} = 1.
$$
This shows that $f$ cannot be a zero function in $L^2(\mu)$.
%
\end{proof}

\begin{example}\label{example013}
Consider $R = 3$, $B = \{0,1,3\}$ and $L = \{0,1\}$. Then $(R,B,L)$ forms a Riesz sequence triple. However, we can never add another $\lambda\subset {\mathbb Z}$ into $L$ so that $(R,B,L\cup\{\lambda\})$ forms a Riesz sequence triple. 
\end{example}

\begin{proof}
Note that the matrix 
$$
{\mathcal F}_{L,B} = \begin{pmatrix}  - & {\bf e}_{R,0}& -\\ - & {\bf e}_{R,1}&-\end{pmatrix} = \frac{1}{\sqrt{3}} \begin{pmatrix}  1 & 1& 1\\ 1 & \omega &1\end{pmatrix}
$$
where  $\omega$ is the cubic root of unity. The two rows are linearly independent. Hence, $(R,B,L)$ forms a Riesz sequence triple. However, to add one more element, we see that we can only add $\lambda = 2 $ (mod 3) by Proposition \ref{Prop_distinct}. However, if we add this, 
$$
{\mathcal F}_{L\cup\{2\},B} =   \frac{1}{\sqrt{3}} \begin{pmatrix}   1 & 1& 1\\ 1 & \omega &1 \\ 1& \omega^2 & 1\end{pmatrix}.
$$
The rows are not linearly independent and do not span ${\mathbb C}^{\#B}$. This example shows that the assumption that elements of $B$ is from a distinct representative in ${\mathbb Z}^d/R({\mathbb  Z}^d)$ is required in Theorem \ref{th4.4}. Otherwise we cannot add another element that preserve the Riesz sequence property. 
\end{proof}

\medskip

\section{Existence of Riesz sequence towers}\label{Section_KS}

In this section, we will show that under some simple assumptions on $R_j,B_j$, one can construct a Riesz sequence tower easily. The key theorem is the following version of the solution to the Kadison-Singer problem, given in \cite[Theorem 6.12]{BCMS16}.

\begin{theorem}\label{thR}{\bf [$R_{\epsilon}$ conjecture]}
Suppose that $\{u_i\}_{i\in I}$ is a unit norm Bessel sequence with Bessel bound $D$ in a separable Hilbert space $H$. Then there exists a universal constant $C_0>0$ such that  for any $\epsilon>0$, one can find a partition $\{I_1,...,I_r\}$ of $I$ of size $r \le C_0(D/\epsilon^4)$ such that  each $\{u_i\}_{i\in I_j}$, $j=1,...,r$ is a Riesz sequence with bounds $1-\epsilon$, $1+\epsilon$.
\end{theorem}

Focusing on the finite dimensional Hilbert space in the above theorem, we notice that if $\epsilon$ is very small, then $r$ will most likely be very large and each $I_j$ may possibly be containing only one element, in which case the theorem would be trivially true. In our application, we will need some $I_j$ to contain more than one element.

\begin{lemma}\label{lemma3.2}
Let $R$ be an integral expanding matrix and let $B$ be a finite subset of $\bz^d$ containing some distinct representatives (mod $R({\mathbb Z}^d)$).  Suppose that $0<\epsilon<1$. Then there exists $L$ with $\#L\ge  \frac{\#B\epsilon^4}{C_0}$ such that $(R,B,L)$ forms a Riesz sequence triple with bounds $1-\epsilon,1+\epsilon$. 
\end{lemma}

\begin{proof}
Given an integer expanding matrix $R$, let $\overline{L}$ be a complete set of distinct representatives of ${\mathbb Z}^d/R^{\mathtt T}({\mathbb Z}^d)$. By Lemma \ref{lemma_tight}, we have that 
$$
\sum_{\lambda\in\overline{L}} \left|\langle{\bf w},{\bf e}_{R,\lambda}  \rangle \right|^2 =  \frac{|\det (R)|}{\#B} \|{\bf w}\|^2
$$
for all ${\bf w}\in{\mathbb C}^{\#B}$.  By Theorem \ref{thR}, one can find a partition $\{I_1,...,I_r\}$ such that each $\{{\bf e}_{R,\lambda}\}_{\lambda\in I_j}$, $j=1,...,r$ is a Riesz sequence with bounds $1-\epsilon$, $1+\epsilon$. Note that at least one of the $I_j$ must contain at least $|\det R|/r$ elements. As $r\le C_0(D/\epsilon^4)$ where $D = |\det(R)|/\#B$, there exists $j$ such that 
$$
\#I_j\ge \frac{\#B\epsilon^4}{C_0}.
$$
Take $L = I_j$ and this completes the proof. 
\end{proof}

\begin{remark}
Indeed, we do not need the full strength of the Kadison-Singer problem (as in the $R_\epsilon$ conjecture above) to prove Lemma \ref{lemma3.2}. It can be also obtained as a consequence of the Bourgain-Tzafriri restricted invertibility theorem \cite{MR890420}, see also \cite[Proposition 4.4]{MR2500595} .  We would like to thank Peter Casazza for pointing this fact to us. 
\end{remark}

\begin{theorem}\label{theorem_Riesz}
Let $\{(R_j,B_j)\}$ be a sequence of pairs with integral expanding matrix $R_j$ and $B_j$  a finite subset containing some distinct representatives (mod $R_j({\mathbb Z}^d)$). Suppose also that the associated fractal measure  
$$
\mu = \delta_{{\bf R}_1^{-1}B_1}\ast\delta_{{\bf R}_2^{-1}B_2}\ast...
$$
has the no-overlap condition. Then $\mu$ admits an exponential Riesz sequence of infinite cardinality.
\end{theorem}

\begin{proof}
First, we note that we can group any $n_j$ consecutive factors without changing the resulting measure $\mu$. 
$$
\begin{aligned}
\mu =& \left(\delta_{{\bf R}_1^{-1}B_1}\ast...\ast\delta_{{\bf R}_{n_1}^{-1}B_{n_1}}\right)\ast\left(\delta_{{\bf R}_{n_1+1}^{-1}B_{n_1+1}}\ast...\ast\delta_{{\bf R}_{n_1+n_2}^{-1}B_{n_2}}\right)\ast...\\
=&\delta_{{\bf R}_{n_1}^{-1}{\bf B}_{n_1}}\ast\delta_{{\bf R}_{n_1+n_2}^{-1}{\bf B}_{n_2}}\ast....
\end{aligned}
$$
 As each $\#B_n\ge 2$,  $\#{\bf B}_n\ge 2^n$. Using this regrouping, we can assume that each $\#B_n\ge 2^n$. For $n$ large enough, we let $1>\epsilon_n> \left(\frac{C_0}{\#B_n}\right)^{1/4}$. As $\#B_n\ge 2^n$, one can also find sequence $\epsilon_n$ that is summable. By Lemma \ref{lemma3.2}, we can find $L_n$ such that $(R_n,B_n,L_n)$ forms a Riesz-sequence triple and 
$\#L_n \ge  \frac{\#B_n\epsilon_n^4}{C_0}>1$. 
Since the sequence $\epsilon_n$ is summable, we can use Theorem \ref{thm1.3}(ii) and  conclude that $\Lambda$ forms a Riesz sequence for $\mu$. Since each $\#L_n>1$, $\Lambda$ is an infinite set. 
\medskip
\end{proof}

\medskip

\subsection{Riesz sequence with optimal Beurling dimension} In this section, we focus on self-affine measures.
\begin{definition}\label{defifs}
For a given expansive $d\times d$ integer matrix $R$ and a finite set of integer vectors $B$ with $\#B$, we define the {\it affine iterated function system} (IFS)
 $$\tau_b(x) = R^{-1}(x+b),\quad ( x\in \br^d, b\in B).$$ The {\it self-affine measure} associated to $R$ and $B$ is the unique probability measure $\mu = \mu(R,B)$ satisfying
\begin{equation}\label{self-affine}
\mu(E) =\frac{1}{\#B} \sum_{b\in B}  \mu (\tau_b^{-1} (E)),\mbox{ for all Borel subsets $E$ of $\br^d$.}
\end{equation}

This measure is supported on the {\it attractor}  $T(R,B)$ which is the unique compact set that satisfies
$$
T(R,B)= \bigcup_{b\in B} \tau_b(T(R,B)).
$$
The set $T(R,B)$ is also called the {\it self-affine set} associated with the IFS. It can also be described as 
$$T(R,B)=\left\{\sum_{k=1}^\infty R^{-k}b_k : b_k\in B\right\}.$$
If $R=\rho O$ where $|\rho|<1$ and $O$ is an orthogonal matrix, then $\mu(R,B)$ is called a self-similar measure.
 One can refer to \cite{Hut81} and \cite{Fal97} for a detailed exposition of the theory of iterated function systems.
\end{definition}

It is known that self-affine measures can be realized as an infinite convolution product
$$
\mu(R,B) = \delta_{R^{-1}B}\ast\delta_{R^{-2}B}\ast.....
$$
Therefore, they fit into the category of measures that we have considered in the previous sections. By factorization of any $n_k$ consecutive factors, we can also write it as 
$$
\mu(R,B) = \delta_{R^{-{n_1}}B_{n_1}}\ast\delta_{R^{-{(n_1+n_2)}}B_{n_2}}\ast.....
$$
It is known \cite[Section 2]{DHL18} that $\mu(R,B)$ satisfies the no-overlap condition if the digits in $B$ are chosen from distinct coset representatives in ${\mathbb Z}^d/R({\mathbb Z}^d)$.
\begin{definition}
Let $\Lambda$ be a countable set on ${\mathbb R}^d$. The $\alpha$-Beurling density is defined to be 
$$
D_{\alpha}^{+}(\Lambda) =  \limsup_{h\rightarrow\infty}\sup_{x \in {\mathbb R}^d} \frac{\#(\Lambda\cap B(x,h))}{h^{\alpha}}, \ B(x,h)= \{y: |y-x|<h\}
$$
and the Beurling dimension of $\Lambda$ is defined to be 
$$
\dim_B(\Lambda) = \sup \{\alpha: D^{\alpha}(\Lambda)>0\}.
$$
\end{definition}

\begin{theorem} \label{thmDHSW} \cite{DHSW11,MR2826404}
{\upshape (i)} Let $\mu(R,B)$ be the self-similar measure with $R = \rho O$ as in Definition \ref{defifs}, and let $E(\Lambda)$ be a Bessel sequence of $\mu(R,B)$. Then the Beurling dimension of $\Lambda$ is at most $\log_{\rho}(\#B)$.

\medskip

{\upshape (ii)} Let $\mu(R,B)$ be the self-similar measure. Then there exists a Bessel sequence $E(\Lambda)$ of positive Beurling dimension.
\end{theorem}

We are going to prove the following theorem:

\begin{theorem} \label{th5.7}
Let $\mu(R,B)$ be the self-similar measure with $R = \rho O$ and assume that $\mu(R,B)$ satisfies the no-overlap condition. Then there exists a Riesz sequence $E(\Lambda)$ of Beurling dimension $\log_{\rho}(\#B)$ for $\mu(R,B)$.
\end{theorem}

\begin{proof}
By regrouping, $\mu(R,B)$ is generated by the tower $(R^{n_k}, B_{n_k})$, where $\#B_{n_k} = (\#B)^{n_k}$.  Let $\epsilon_k = \left(\frac{C_0}{n_k^5}\right)^{1/4}$ and by Lemma \ref{lemma3.2}, we can find $L_{n_k}$ such that
$(R^{n_k}, B_{n_k}, L_{n_k})$ forms a Riesz sequence triple and 
\begin{equation}
\#L_{n_k}\ge \frac{(\#B)^{n_k}}{n_k^5}.
\label{eqdl}
\end{equation}

We now take $n_k = k$ 
so that $\sum \epsilon_k<\infty$ and then we obtain a Riesz sequence tower $(R^{n_k}, B_{n_k}, L_{n_k})$ and hence $\mu(R,B)$ admits a Riesz sequence $E(\Lambda)$, with 
$$
\Lambda = \bigcup_{k=1}^{\infty}{\bf L_k}, \ {\bf L}_{k}: =  L_{n_1}+(R^{\mathtt T})^{n_1} L_{n_2}+(R^{\mathtt T})^{n_1+n_2} L_{n_2}+...+(R^{\mathtt T})^{n_1+...+n_{k-1}} L_{n_k}
$$
by Theorem \ref{thm1.3}(ii). It remains to show that the Beurling dimension of $\Lambda$ is $\log_{\rho}(\#B)$.

\medskip

Let $Q_n = (R^{\mathtt T})^n(\cj B_{\sqrt{d}}(0))=\cj B_{\rho^n\sqrt{d}}(0)$ for $n=1,2,...$, note that $B_{\sqrt d}(0)$ contains the cube $[-1/2,1/2]^d$. We note also that $L_{n_i}$ can be chosen to be inside $Q_{n_i}$, by reducing $(\mod R^{\mathtt T})^{n_i}$, and then ${\bf L}_{k}$ is inside $Q_{n_1+...+n_k+1}$. For any ${\bf x}\in {\bf L}_k$, 
$$
|{\bf x}|\le   \sqrt{d}\rho^{n_1+...+n_k+1}.
$$
Therefore, letting $h_k = \sqrt{d}\rho^{n_1+...+n_k+1} =\sqrt{d}\rho^{k(k+1)/2+1}$ (since $n_k = k$), and using \eqref{eqdl}, we have 
$$
\#(\Lambda\cap B_{h_k}(0)) \ge \prod_{i=1}^k\#L_{n_i}\ge \prod_{i=1}^k\frac{(\#B)^{n_i}}{n_i^5} = \frac{(\#B)^{k(k+1)/2}}{(k!)^5}
$$
Let $\alpha = \log_{\rho}(\#B)$. For any $\eta>0$, 
$$
\begin{aligned}
\frac{\#(\Lambda\cap B_{h_k}(0))}{h_k^{\alpha-\eta}} \ge& \frac{\frac{(\#B)^{k(k+1)/2}}{(k!)^5}}{d^{(\alpha-\eta)/2}(\#B)^{k(k+1)/2+1}\rho^{-\eta (k(k+1)/2+1)}}\\
\ge&\frac{1}{d^{(\alpha-\eta)/2}} \cdot \frac{\rho^{\eta k(k+1)/2}}{k^{5k}}, \ \\
\ge&\frac{1}{d^{(\alpha-\eta)/2}} \cdot e^{ck^2-5k\ln k} \rightarrow +\infty, \ (\mbox{let} \ c = \ln(\rho^{\eta/2})>0)\\
\end{aligned}
$$
since $\lim_{k\rightarrow \infty}(ck^2-5k\ln k) = +\infty$. This shows that $D^{+}_{\alpha-\eta}(\Lambda) = \infty$ for all $\eta>0$. Hence, dim$_B(\Lambda)\ge \log_{\rho}(\#B)$. Hence, together with Theorem \ref{thmDHSW}, we have dim$_B(\Lambda)= \log_{\rho}(\#B)$. This completes the proof. 

\end{proof}

\smallskip

\begin{acknowledgements}
The authors would like to thank the referee for his/her valuable suggestion. They would also like to thank the referee and professor Peter  Casazza for pointing out how Lemma \ref{lemma3.2} can also be deduced from a weaker version of the Bourgain-Tzafriri theorem. This work was partially supported by a grant from the Simons Foundation (\#228539 to Dorin Dutkay).
\end{acknowledgements}

\bibliographystyle{alpha}
\bibliography{eframes}
\end{document}